\documentclass{amsart}
\usepackage{}
\usepackage{amssymb}
\usepackage{amsmath}
\usepackage{cases}
\usepackage{mathrsfs}
\usepackage{amsfonts}
\usepackage{cite}

\newtheorem{theorem}{Theorem}[section]
\newtheorem{lemma}[theorem]{Lemma}

\theoremstyle{definition}

\newtheorem{corollary}[theorem]{Corollary}
\theoremstyle{remark}

\numberwithin{equation}{section}

\begin{document}
\title{Representations for generalized Drazin inverse of operator matrices over a Banach space}

\author{Daochang Zhang}
\address{School of Mathematics, Jilin University, Changchun 130012, China }
%\curraddr{}
\email{daochangzhang@gmail.com}
\thanks{Supported by NSFC (No. 11371165).}
%    General info

\date{December 6, 2014.}


\begin{abstract}
 In this paper we give expressions for the generalized Drazin inverse of a (2,2,0) operator matrix and a $2\times2$ operator matrix under certain circumstances, which generalizes and unifies several results in the literature.
\end{abstract}

\maketitle
\section{introduction}

The concept of the generalized Drazin inverse (GD-inverse) in a Banach algebra was introduced by Koliha \cite{Koliha1996}. Let $\mathcal {B}$ be a complex unital Banach algebra. An element $a$ of $\mathcal {B}$ is generalized Drazin invertible in case there is an element $b\in\mathcal {B}$ satisfying
\begin{equation*}\label{defingdrazin}
  ab=ba,~~~ bab=b,~~~\text{and}~~~ a-a^{2}b~~~\text{is quasinilpotent}.
\end{equation*}
Such $b$, if it exists, is unique; it is called a generalized Drazin inverse of $a$, and will be
denoted by $a^d$. Then the spectral idempotent $a^\pi$ of $a$ corresponding to {0} is given by $a^\pi=1-aa^d$.

The  GD-inverse was extensively investigated for  matrices over complex Banach algebras and  matrices of
bounded linear operators over complex Banach spaces.
  The GD-inverse of the operator matrix has various applications in singular differential equations and singular difference equations, Markov chains and iterative methods and so on (see
\cite{Ben2006,Campbell1980I,Campbell1980II,Campbell1983,Castro2006,Castro2002,Djordjev2001,Du2005,Wei2005,Wei2006(1),Wei2006(2)}).

The generalized Drazin inverse is a generalization of Drazin inverses and group inverses.
The study on representations for the Drazin inverse of block matrices essentially originated from finding the general
expressions for the solutions to singular systems of differential equations \cite{Campbell1983,Campbell1991,Campbell1976}.
Until now, there have been many formulae for the Drazin inverse of general $2\times2$ block matrices under some restrictive assumptions
(see \cite{Cvetkov2008,Djordjev2001,Dopazo2010,guophd,Guo2010,Hartwig1977,Hartwig2006,Liu2012,Meyer1997}).
%\begin{enumerate}
%\item[(1)] $BC =0, BD = 0$ and $DC = 0$ (see \cite{Djordjev 2001});
%\item[(2)] $BC =0, DC = 0$ (or $BD = 0$) and $D$ is nilpotent (see \cite{Hartwig2006});
%\item[(3)] $BC = 0$ and $DC = 0$ (see \cite{Cvetkov2008});
%\item[(4)] $BC =0, BDC = 0$ and $BD^2 = 0$ (see \cite{Dopazo2010});
%\item[(5)] $BD^iC=0$, for $i=0,1,...,n-1$ (see\cite{Guo2010});
%\item[(6)] $ABC=0$, and $BD^iC=0$ for $i=1,2,...,n-1$ (see\cite{guophd}).
%\end{enumerate}

Some results of the Drazin inverse have been developed in the GD-inverse of operator matrices over Banach spaces (see
\cite{Castro2009,CvetkovAS2011,Deng2010,Djordjev2001,guophd}).
Assume that both $X$ and $Y$ are complex Banach spaces. Denote by $\mathcal {B}(X,Y)$ the set of all bounded linear operators from $X$ to $Y$, and write $\mathcal {B}(X,X)=\mathcal {B}(X)$.
Let an operator matrix $M=\begin{pmatrix}
       A&B\\
       C&D
      \end{pmatrix},$
where $A\in\mathcal {B}(X),D\in\mathcal {B}(Y),B\in\mathcal {B}(Y,X),C\in\mathcal {B}(X,Y)$.

Castro-Gonz\'{a}lez \cite{Castro2009} derived explicit expressions of the GD-inverse of $M$ under certain conditions, which extended some results of \cite{Deng2010,Djordjev2001}.
Cvetkovi\'{c} \cite{CvetkovAS2011} also extended some results of \cite{Deng2010,Djordjev2001}.
Recently, Mosi\'{c} \cite{Mosic2014} gave the new formulae for the GD-inverse of $2\times 2$ matrices in a Banach algebra.

%\begin{enumerate}
%\item[(1)] $bc =0, bd = 0$ and $dc = 0$ (see \cite{Djordjev 2001});
%\item[(2)] $bc =0, bd = 0$ (see \cite{Deng2010});
%\item[(3)] $bc =0, dc = 0$ (see \cite{Deng2010});
%\item[(4)]$bca=0,bd=0,dc=0$ (see \cite{Castro2009});
%\item[(5)]$abc=0,bd=0,dc=0$ (see \cite{Cvetkov AS2010});
%\item[(6)]$bca=0,bd=0,$ and $bc$ is nilpotent (see \cite{Castro2009});
%\item[(7)]$abc=0,dc=0,$ and $bc$ is nilpotent (see \cite{Cvetkov AS2010});
%\item[(8)]$bca=0,dc=0,$ and $d$ is nilpotent (see \cite{Castro2009});
%\item[(9)]$abc=0,dc=0,$ and $d$ is nilpotent (see \cite{Cvetkov AS2010});
%\item[(10)] $bd^d=0$, $bd^ic=0,$ for $i=0,1,...,n-1$ (see\cite{guophd});
%\item[(11)] $d^dc=0$, $bd^ic=0,$ for $i=0,1,...,n-1$ (see\cite{guophd});
%\item[(12)] $bd^d=0$, $abc=0$, and $bd^ic=0$ for $i=1,2,...,n-1$ (see\cite{guophd}).
%\item[(13)] $d^dc=0$, $abc=0$, and $bd^ic=0$ for $i=1,2,...,n-1$ (see\cite{guophd}).
%\end{enumerate}

In this paper, we derive new formulae for the GD-inverse of a (2,2,0) operator matrix $N$ under certain circumstances.
Furthermore,
we apply $N^d$ to give representations of $M^d$ under weaker restrictions, which generalizes and unifies several results of \cite{Castro2009,CvetkovAS2011,Deng2010,Djordjev2001,guophd,Mosic2014}.

Since $(a^d)^n=(a^n)^d$ for any $a\in \mathcal {B}$
we adopt the convention that  $a^{dn}=(a^d)^n$ and $a^0=1$
 and $\sum_{i=0}^k\ast=0$ in case $k<0$.
Moreover, we define
$\begin{pmatrix}
  a&b\\c&d
\end{pmatrix}^\top=\begin{pmatrix}
  a&c\\b&d
\end{pmatrix}$ for $a,b,c,d\in \mathcal {B}.$
%For notational convenience, we write $A^e$ for $AA^d$.

\section{Generalized Drazin inverse of a (2,2,0) operator matrix}
Let $\mathcal {B}$ be a complex unital Banach algebra.
An element $a\in\mathcal {B}$ is called quasinilpotent, if $\lim_{n\rightarrow\infty}\|a^n\|^{\frac{1}{n}}=0.$
Let $\mathcal {M}_2(\mathcal {B})$ be the $2\times2$ matrix algebra over $\mathcal {B}$. Given an idempotent $e$ in $\mathcal {B}$, we consider the set
$\mathcal {M}_2(\mathcal{B},e)=\begin{pmatrix}
  e\mathcal{B}e&e\mathcal{B}(1-e)\\
  (1-e)\mathcal{B}e&(1-e)\mathcal{B}(1-e)
\end{pmatrix}\subset \mathcal {M}_2(\mathcal{B})$.
Then $\mathcal {M}_2(\mathcal {B},e)$ is a unital Banach algebra with respect to the norm
$$\left\|\begin{pmatrix}
  a_{11}&a_{12}\\a_{21}&a_{22}
\end{pmatrix}\right\|=\|a_{11}+a_{12}+a_{21}+a_{22}\|.$$
\begin{lemma}\label{lempierce}
Let $e$ be an idempotent of $\mathcal {B}$. For any $a\in\mathcal{B}$ let
$$\sigma(a)
=\begin{pmatrix}
  eae&ea(1-e)\\
  (1-e)ae&(1-e)a(1-e)
\end{pmatrix}
\in\mathcal {M}_2(\mathcal{B},e).$$
Then the mapping $\sigma$ is an isometric Banach algebra isomorphism from
$\mathcal{B}$ to $\mathcal {M}_2(\mathcal{B},e)$ such that
\begin{enumerate}
\item
$(\sigma(a))^d=\sigma(a^d)$;\\
\item if $(\sigma(a))^d=\begin{pmatrix}
     \alpha & \beta \\
     \gamma & \delta
   \end{pmatrix}$, then $a^{d}=\alpha+\beta+\gamma+\delta$.
\end{enumerate}
\end{lemma}
\begin{proof}
By \cite [Lemma 2.1]{Castro2004} we have the mapping $\sigma$ is an isometric Banach algebra isomorphism from
$\mathcal{B}$ to $\mathcal {M}_2(\mathcal{B},e)$. The rest of the proof is obvious.
\end{proof}

\begin{lemma}\label{triangle}(\cite{Djordjev2001})
  Let $x=\begin{pmatrix}
                   a & b \\
                   0 & d
                 \end{pmatrix}$
     and let
     $y=\begin{pmatrix}
                   d & 0 \\
                   b & a
                 \end{pmatrix}$~~~for $a,b,d\in\mathcal {B}$. Then
  $$x^{d}=\begin{pmatrix}
                   a^{d} & X \\
                   0 & d^{d}
                 \end{pmatrix}, \quad
  y^{d}=\begin{pmatrix}
                   d^{d} & 0 \\
                   X & a^{d}
                 \end{pmatrix},$$
   where
\begin{equation*}
   X=a^{\pi}\sum_{i=0}^{\infty}a^{i}bd^{d(i+2)}+\sum_{i=0}^{\infty}a^{d(i+2)}bd^{i}d^{\pi}-a^{d}bd^{d}.
\end{equation*}
\end{lemma}

\begin{lemma}\label{(abpi)d=adbpi}
Let $e$ be an idempotent of $\mathcal {B}$ and let $a\in \mathcal{B}$ be generalized Drazin invertible such that $ea(1-e)=0$. Then $ea$ and $a(1-e)$ are both generalized Drazin invertible, and
 $$(ea)^d=ea^d, \quad (a(1-e))^d=a^d(1-e),\quad(ea)^n=ea^n$$
 for any positive integer $n$.
\end{lemma}
\begin{proof}
 Since $ea(1-e)=0$, combining Lemma \ref{lempierce}  and Lemma \ref{triangle}, we have $ea^d(1-e)=0$. Then
$eae a^d =eaa^d =ea^de a $ \text{and} $ea^de a  ea^d =ea^d $.
Furthermore,
\begin{multline*}
   \lim_{n\rightarrow\infty}\|(ea -(ea )^2ea^d )^n\|^{\frac{1}{n}}=\lim_{n\rightarrow\infty}\|(eaa^\pi )^n\|^{\frac{1}{n}}\\
=\lim_{n\rightarrow\infty}\|ea^n a^\pi  \|^{\frac{1}{n}}
\leq \lim_{n\rightarrow\infty}\|e \|^\frac{1}{n}\|a^n a^\pi \|^\frac{1}{n} =0.
\end{multline*}
Hence $ea $ is generalized Drazin invertible and $(ea )^d=ea^d $.
Similarly we can prove that $ a (1-e)$ is generalized Drazin invertible and $( a(1-e))^d= a^d(1-e)$.
Using $ea(1-e)=0$ we get easily $(ea)^n=ea^n$ for any positive integer $n$.
\end{proof}

\begin{lemma}\label{cline}(\cite{cline2014})(Cline's Formula)
For $a,b\in\mathcal {B}$, $ab$ is generalized Drazin invertible if and only if so is $ba$. Furthermore, if $ab$ is generalized Drazin invertible, then
$$(ba)^d = b(ab)^{2d}a.$$
\end{lemma}

The following lemma is an immediate corollary of \cite[Corollary 3.3.7]{guophd}.
\begin{lemma}\label{guo3.3.7cor}
Let $x=\begin{pmatrix}
  a&b\\
  c&d
\end{pmatrix}\in\mathcal {M}_2(\mathcal {B})$ with $a$ and $d$  generalized Drazin invertible.
If $abc=0,~bd=0,~(bc)^{d}=0$, then
  $x$ is generalized Drazin invertible, and
 $$x^d=
 \begin{pmatrix}
   \phi_{1}a                        & \phi_{1}b \\
   \tau a+\psi_{1} &  d^{d}+\tau b
 \end{pmatrix},$$
 where
\begin{equation*}
\begin{array}{ll}
   \phi_{n}&=\Sigma_{j=0}^{\infty}(bc)^{j}a^{d(2j+2n)},\\
   \psi_{n}&=\Sigma_{j=0}^{\infty}d^{d(2j+2n)}(cb)^{j}c,\\
   \tau&=\Sigma_{i=0}^{\infty}(cb+d^2)^{i}ca^{d(2i+3)}+\Sigma_{i=0}^{\infty}d^{\pi}d^{2i+1}c\phi_{i+2}\\
    &\quad-\Sigma_{i=0}^{\infty}d^2(cb+d^2)^i\psi_1a^{d(2i+3)}+\Sigma_{i=0}^{\infty}\psi_{i+2}a^{2i+1}a^\pi\\
   &\quad+\Sigma_{i=0}^{\infty}d^{d(2i+3)}c(a^2+bc)^ia^\pi-\Sigma_{i=0}^{\infty}d^{d(2i+1)}c(bc)^i\phi_1-\psi_1a^d.
\end{array}
\end{equation*}
\end{lemma}

Now we can give our first main result. Recall that $a^{dn}=(a^d)^n$
and $\begin{pmatrix}
  a&b\\c&d
\end{pmatrix}^\top=\begin{pmatrix}
  a&c\\b&d
\end{pmatrix}$ for $a,b,c,d\in \mathcal {B}.$
\begin{theorem}\label{th220babpi=0}
Let $N=\begin{pmatrix}
        E & I \\
        F & 0
      \end{pmatrix}$
 be an operator matrix with $E$ and $F$ generalized Drazin invertible. If $F^dEF^\pi=0$ and $F^\pi FE=0$,  then $N$ is generalized Drazin invertible,
and
\begin{align*}
N^d&=\sum\limits_{i=0}^\infty\begin{pmatrix}
E^{2i+1}E^\pi F^\pi EF^d & 0   \\
E^{2i}E^\pi( F^\pi-E F^\pi EF^d)EF^d & 0
\end{pmatrix}^\top\begin{pmatrix}
0 & F^d \\
I & -EF^d
\end{pmatrix}^{2i+1}\\
&\quad+\begin{pmatrix}
    E^dF^\pi+\sum_{i=0}^\infty E^{d(2i+3)}F^\pi F^{i+1} & FF^d \\
    F^d-E^dF^\pi EF^d+\sum_{i=0}^\infty E^{d(2i+2)}F^\pi F^i & -FF^dEF^d
  \end{pmatrix}^\top.
\end{align*}
\end{theorem}
\begin{proof}
We adopt the convention that $F^e=FF^d$.
Let $e=\begin{pmatrix}
          F^e & 0 \\
          0   & I
        \end{pmatrix}$, $\sigma$ as in Lemma \ref{lempierce}, and $\sigma(N) =\begin{pmatrix}
      a & b \\
      c & d
    \end{pmatrix}$. Since $F^dEF^\pi=0$,
 we have
\begin{equation*}\label{eqa1}
a=\begin{pmatrix}
    F^eE & F^e \\
    FF^e    & 0
  \end{pmatrix},~
~b=\begin{pmatrix}
    0       & 0 \\
    FF^\pi    & 0
  \end{pmatrix},
~~c=\begin{pmatrix}
    F^\pi EF^e & F^\pi \\
    0    & 0
  \end{pmatrix},
~~d=\begin{pmatrix}
    EF^\pi &  0 \\
    0      & 0
  \end{pmatrix}.
\end{equation*}
Note that  $a$ has the group inverse
\begin{equation}\label{eqa3}
 a^\sharp=\begin{pmatrix}
             0 & F^d \\
             F^e & -F^{e}EF^{d}
           \end{pmatrix}
\end{equation}
and so $aa^\pi=0$. Using Lemma \ref{(abpi)d=adbpi} we have $(EF^\pi)^d=E^dF^\pi,$
and
$$(a^\sharp)^n
=\left(\begin{pmatrix}
  F^e&0\\
  0&F^e
\end{pmatrix}
\begin{pmatrix}
  0&F^d\\
  I&-EF^d
\end{pmatrix}\right)^n
=\begin{pmatrix}
  F^e&0\\
  0&F^e
\end{pmatrix}
\begin{pmatrix}
  0&F^d\\
  I&-EF^d
\end{pmatrix}^n$$
for any positive integer $n$. Hence
\begin{equation}\label{d^d}
   d^d=\begin{pmatrix}
    E^dF^\pi &  0 \\
    0      & 0
  \end{pmatrix}.
\end{equation}Note that $(F^\pi F)^d=0$,
and so
\begin{equation}\label{eqa2}
bc=\begin{pmatrix}
            0 & 0 \\
            F^\pi FE & F^\pi F
          \end{pmatrix}
 ~~~~\text{and}~~~~ cb=\begin{pmatrix}
                      F^\pi F & 0 \\
                      0 & 0
                    \end{pmatrix}.
\end{equation}
Using Lemma \ref{triangle} we get
$(bc)^d=(cb)^d=0$. Since
$F^dE^{i+1}F^\pi=F^d(EF^\pi)^{i+1}=0$
for any nonnegative integer $i$, we have
\begin{equation*}
ab=0,~~bd=0,~~bca=0.
\end{equation*}
By Lemma \ref{guo3.3.7cor}
 we have
\begin{equation}\label{eqa4}
   (\sigma(N))^d=
 \begin{pmatrix}
   a^d        & 0 \\
   \Sigma_{0} &  d^{d}+\Lambda
 \end{pmatrix},
\end{equation}
 where
 \begin{equation*}
   \begin{array}{ll}
   \Lambda=&\sum_{i=0}^{\infty}d^{d(2i+3)}c(bc)^ib,\\
   \Sigma_{0}=&\sum_{i=0}^{\infty}d^{2i}ca^{d(2i+2)}+\sum_{i=0}^{\infty}d^{\pi}d^{2i+1}ca^{d(2i+3)}
    -\sum_{i=0}^{\infty}d^{2i+1}d^dca^{d(2i+2)}\\
    &+\sum_{i=0}^{\infty}d^{d(2i+2)}(cb)^ic-d^dca^d-d^{2d}ca^da.
   \end{array}
 \end{equation*}
Substituting \eqref{eqa3}, \eqref{d^d} and \eqref{eqa2} into \eqref{eqa4} and using Lemma \ref{lempierce} will give the expression of $N^d$ that we wanted.
%\begin{equation*}
%\begin{array}{ll}
%\Lambda
%=&\begin{pmatrix}
%    \sum_{i=0}^\infty(E^d)^{2i+3}F^\pi F^{i+1} & 0 \\
%    0 & 0
%  \end{pmatrix},\\
%\Sigma_{0}=&\sum\limits_{i=0}^\infty\begin{pmatrix}
%                                         E^{2i+1}E^\pi F^\pi EF^d & E^{2i}E^\pi F^\pi EF^d-E^{2i+1}E^\pi F^\pi (EF^d)^2   \\
%                                         0 & 0
%                                       \end{pmatrix}\begin{pmatrix}
%                                                           0 & F^d \\
%                                                           F^e & -F^eEF^d
%                                                         \end{pmatrix}^{2i+1}\\
%&+\begin{pmatrix}
%    0 & -E^dF^\pi EF^d+\sum_{i=0}^\infty(E^d)^{2i+2}F^\pi F^i \\
%    0 & 0
%  \end{pmatrix}.
%\end{array}
%\end{equation*}
%
%Obviously, The result only needs routine computations.
\end{proof}

\begin{corollary}\label{corNdn}
Let $N=\begin{pmatrix}
        E & I \\
        F & 0
      \end{pmatrix}$
 be an operator matrix with $E$ and $F$ generalized Drazin invertible. If $F^dEF^\pi=0$ and $F^\pi FE=0$,  then $N$ is generalized Drazin invertible,
and
\begin{align*}
 N^{dn}&=\sum\limits_{i=0}^\infty\begin{pmatrix}
   E^{2i+1}E^\pi F^\pi EF^d& 0   \\
   E^{2i}E^\pi(F^\pi-EF^\pi EF^d)EF^d & 0
 \end{pmatrix}^\top
 \begin{pmatrix}
    0&F^d\\
   I&-EF^d
    \end{pmatrix}^{2i+n}\\
&\quad+\sum\limits_{j=1}^n\begin{pmatrix}
    E^{dj}F^\pi+\sum_{i=0}^\infty E^{d(2i+j+2)}F^{i+1}F^\pi      & 0 \\
    E^{dj}(\sum_{i=0}^\infty E^{d(2i+1)}F^{i}F^\pi-F^\pi EF^d)& 0
  \end{pmatrix}^\top
  \begin{pmatrix}
                      0 & F^d \\
                      FF^d & -FF^dEF^d
                    \end{pmatrix}^{n-j}\\
&\quad+\begin{pmatrix}
                 0 & F^d \\
                 FF^d & -FF^dEF^d
               \end{pmatrix}^n
\end{align*}
for any positive integer $n$.
\end{corollary}
\begin{proof}
Let $N^d=P+Q+R$ by Theorem \ref{th220babpi=0}, where
\begin{align*}
P&=\sum\limits_{i=0}^\infty\begin{pmatrix}
E^{2i+1}E^\pi F^\pi EF^d & 0   \\
E^{2i}E^\pi( F^\pi-E F^\pi EF^d)EF^d & 0
\end{pmatrix}^\top\begin{pmatrix}
0 & F^d \\
I & -EF^d
\end{pmatrix}^{2i+1},\\
Q&=\begin{pmatrix}
    E^dF^\pi+\sum_{i=0}^\infty E^{d(2i+3)}F^\pi F^{i+1} & 0\\
    -E^dF^\pi EF^d+\sum_{i=0}^\infty E^{d(2i+2)}F^\pi F^i & 0
  \end{pmatrix}^\top,\\
R&=\begin{pmatrix}
  0&F^d\\
  FF^d&-FF^dEF^d
\end{pmatrix}.
\end{align*}
Since $F^dE^{i+1}F^\pi=F^d(EF^\pi)^{i+1}=0$ for any nonnegative integer $i$,
and since $F^\pi FE=0$ we have $P^2=0$,~$ RP=0$,~$RQ=0$,~$PQ=QP=0$, {and}
\begin{align*}
Q^n=\sum\limits_{j=1}^n\begin{pmatrix}
    E^{dn}F^\pi+\sum_{i=0}^\infty E^{d(2i+n+2)}F^{i+1}F^\pi      & 0 \\
    E^{dn}(\sum_{i=0}^\infty E^{d(2i+1)}F^{i}F^\pi-F^\pi EF^d)& 0
  \end{pmatrix}^\top
\end{align*}
for any positive integer $n$. Then $N^{dn}=Q^n+R^n+PR^{n-1}+\sum_{j=1}^{n-1}Q^jR^{n-j}$, by a routine computation, we get the expression of $N^{dn}$ as shown in Corollary \ref{corNdn}.
\end{proof}

The following theorem, which is a dual version of Theorem \ref{th220babpi=0}, can be proved similarly.

\begin{theorem}\label{th220bpiab=0}
Let  $N=\begin{pmatrix}
        E & I \\
        F & 0
      \end{pmatrix}$
be an operator matrix with $E$ and $F$ generalized Drazin invertible. If $F^\pi EF^d=0$ and $EFF^\pi=0$, then $N$ is generalized Drazin invertible,
and
\begin{align*}
N^d&=\sum\limits_{i=0}^\infty\begin{pmatrix}
                                        0 & F^d \\
                                        FF^d & -EF^d
                                      \end{pmatrix}^{2i+2}
                                      \begin{pmatrix}
                                      F^dEF^\pi E^{2i+2}E^\pi & F^dEF^\pi E^{2i+1}E^\pi \\
                                      GE & G
                                      \end{pmatrix}\\
   &\quad+\begin{pmatrix}
      \sum\limits_{i=0}^\infty F^iF^\pi E^{d(2i+1)}+F^dEF^\pi E^\pi & \sum\limits_{i=0}^\infty F^iF^\pi E^{d(2i+2)}+F^d-F^dEF^\pi E^d \\
      H+FF^d-EF^dEF^\pi
      &HE^d+EF^dEF^\pi E^d-EF^d
    \end{pmatrix}
\end{align*}
such that
$$G=(FF^dEF^\pi-EF^dEF^\pi E)E^{2i}E^\pi,$$
and
$$H=\sum_{i=1}^\infty F^{i+1}F^\pi E^{d(2i+2)}-FF^dEF^\pi E^d.$$
\end{theorem}
%&+\begin{pmatrix}
%      F^\pi E^d+H+F^dEF^\pi E^\pi & F^\pi E^{2d}+HE^d+F^d-F^dEF^\pi E^d \\
%      FHE^d-F^eEF^\pi E^d+F^e-EF^dEF^\pi
%      &FHE^{2d}-F^eEF^\pi E^{2d}+EF^dEF^\pi E^d-EF^d
%    \end{pmatrix},
%\end{array}
%\end{equation*}
%where
%$$G=(F^eEF^\pi-EF^dEF^\pi E)E^{2i}E^\pi,$$
%and
%$$H=\sum_{i=1}^\infty F^{i}F^\pi (E^d)^{2i+1}.$$
%\end{theorem}
%
\begin{proof}
We adopt the convention that $F^e=FF^d$.
 Let $e=\begin{pmatrix}
          F^\pi & 0 \\
          0   & 0
        \end{pmatrix}$.
Using Lemma \ref{(abpi)d=adbpi} we have $(EF^\pi)^d=E^dF^\pi.$
The proof is similar in spirit to that of Theorem \ref{th220babpi=0}.
Since $F^\pi EF^d=0$, we have
\begin{equation*}
a=\begin{pmatrix}
    F^\pi E &  0 \\
    0      & 0
  \end{pmatrix},~~
b=\begin{pmatrix}
    0 & F^\pi \\
    0    & 0
  \end{pmatrix},~~
c=\begin{pmatrix}
    F^eEF^\pi       & 0 \\
    FF^\pi    & 0
  \end{pmatrix},~~
d=\begin{pmatrix}
    EF^e & F^e \\
    FF^e    & 0
  \end{pmatrix}.
\end{equation*}
Then
\begin{equation}\label{eqb1}
(bc)^d=0=(cb)^d,\quad
  d^\sharp=\begin{pmatrix}
             0 & F^d \\
             F^e & -EF^{d}
           \end{pmatrix}~~~~\text{and}~~~~
 dd^\pi=0. \
\end{equation}
Since $F^\pi E^{i+1}F^d=(F^\pi E)^{i+1}F^d=0$ for any nonnegative integer $i$, we have
\begin{equation}\label{eqb2}
abc=0,~~bd=0,~~dcbc=0.
\end{equation}
Combining \eqref{eqb1} and \eqref{eqb2}, in a way exactly similar to Theorem \ref{th220babpi=0}, we get the result.
\end{proof}

\section{Applications to a $2\times2$ operator matrix}
\begin{lemma}\cite[Theorem 3.2.2]{guophd}\label{guoth3.2.2}
Let
$x=\begin{pmatrix}
      a & b \\
      c & d
    \end{pmatrix}\in\mathcal {M}_2(\mathcal {B})$
with $a$ and $d$ generalized Drazin invertible. If $ca^d=0$ and $ca^ib=0$ for any nonnegative integer $i$, then
  $x$ is generalized Drazin invertible, and
\begin{equation*}
 x^d=\begin{pmatrix}
         a^d+\varphi & \phi \\
         \psi & d^d
       \end{pmatrix},
\end{equation*}
where
\begin{equation*}
\begin{array}{ll}
 \psi=&\sum_{i=0}^\infty d^{d(i+2)}ca^i,\\
    \phi=&a^\pi\sum_{i=0}^\infty a^ibd^{d(i+2)}+\sum_{i=0}^\infty a^{d(i+2)}bd^id^\pi-a^dbd^d,\\
    \varphi=&a^\pi\sum_{i=0}^\infty\sum_{j=0}^\infty a^ibd^{d(i+j+3)}ca^j
       -\sum_{i=0}^\infty\sum_{j=0}^\infty a^{d(i+1)}bd^id^{d(j+2)}ca^j\\
       &+\sum_{i=0}^\infty\sum_{j=0}^ia^{d(i+3)}bd^jca^{i-j}.
\end{array}
\end{equation*}
\end{lemma}

Now, based on an observation on the matrix decomposition, we apply the representations of generalized Drazin inverse of the (2,2,0) operator matrix to give our another main result. Recall that $a^e=a^da$, $a^{dn}=(a^d)^n$
and $\begin{pmatrix}
  a&b\\c&d
\end{pmatrix}^\top=\begin{pmatrix}
  a&c\\b&d
\end{pmatrix}$ for $a,b,c,d\in \mathcal {B}.$

\begin{theorem}\label{BDd=0,FEFpi=0}
 Let  $M=\begin{pmatrix}
         A & B \\
         C & D
       \end{pmatrix}$ be an operator matrix
with $A$ and $D$ generalized Drazin invertible. If $(BC)^dA(BC)^\pi=0$, $(BC)^\pi BCA=0$, $BD^d=0$, and $BD^iC=0$ for any positive integer $i$,
then
  $M$ is generalized Drazin invertible, and
\begin{align*}
                           M^{d}=&\begin{pmatrix}
    \begin{pmatrix}
       0 \\
       D^e+D\Omega
     \end{pmatrix}+\begin{pmatrix}
                          A & I \\
                          C & 0
                        \end{pmatrix}\Phi
  \end{pmatrix}\begin{pmatrix}
    \begin{pmatrix}
       0 &D^d+\Omega
     \end{pmatrix}+\Psi\begin{pmatrix}
                          I & 0 \\
                          0 & B
                        \end{pmatrix}
  \end{pmatrix}\\
&+\begin{pmatrix}
    \begin{pmatrix}
       0 \\
       D\Psi
     \end{pmatrix}+\begin{pmatrix}
                          A & I \\
                          C & 0
                        \end{pmatrix}\Delta^d
  \end{pmatrix}\begin{pmatrix}
    \begin{pmatrix}
       0 &
    \Phi
     \end{pmatrix}+\Delta^d\begin{pmatrix}
                          I & 0 \\
                          0 & B
                        \end{pmatrix}
  \end{pmatrix},
    \end{align*}
where
\begin{align*}
\Phi &=\sum\limits_{k=0}^\infty\Delta^{d({k+2})}\begin{pmatrix}
0 \\ BD^{k+1}
\end{pmatrix}, \\
\Psi &=\sum\limits_{k=0}^\infty\begin{pmatrix}
D^{k}D^\pi C & 0
\end{pmatrix}\Delta^{d(k+2)}
+\sum\limits_{k=0}^\infty\begin{pmatrix}
D^{d(k+2)}C & 0
\end{pmatrix}\Delta^k\Delta^\pi-\begin{pmatrix}
D^dC & 0
\end{pmatrix}\Delta^d,
\\
\Omega &=\sum\limits_{k=0}^\infty\sum\limits_{l=0}^\infty\begin{pmatrix}
D^kD^\pi C & 0
\end{pmatrix}\Delta^{d(k+l+3)}\begin{pmatrix}
0 \\ BD^{l+1}
\end{pmatrix}\\
&\quad-\sum\limits_{k=0}^\infty\sum\limits_{l=0}^\infty\begin{pmatrix}
D^{d(k+1)}C & 0
\end{pmatrix}\Delta^k\Delta^{d(l+2)}\begin{pmatrix}
0 \\ BD^{l+1}
\end{pmatrix}\\
&\quad+\sum\limits_{k=0}^\infty\sum\limits_{l=0}^k\begin{pmatrix}
D^{d(k+3)}C & 0
\end{pmatrix}\Delta^l\begin{pmatrix}
0 \\ BD^{k-l+1}
\end{pmatrix},\\
 \Delta^{dn} &=\sum\limits_{i=0}^\infty\begin{pmatrix}
   A^{2i+1}A^\pi(BC)^\pi A(BC)^d& 0   \\
   A^{2i}A^\pi((BC)^\pi-A(BC)^\pi A(BC)^d)A(BC)^d & 0
 \end{pmatrix}^\top   \\
 &\quad\qquad\times\begin{pmatrix}
    0&(BC)^d\\
   I&-A(BC)^d
    \end{pmatrix}^{2i+n}\\
&\quad+\sum\limits_{j=1}^n\begin{pmatrix}
    A^{dj}(BC)^\pi+\sum_{i=0}^\infty A^{d(2i+j+2)}(BC)^{i+1}(BC)^\pi      & 0 \\
    \sum_{i=0}^\infty A^{d(2i+j+1)}(BC)^{i}(BC)^\pi-A^{dj}(BC)^\pi A(BC)^d& 0
  \end{pmatrix}^\top\\
  &\quad\qquad\times\begin{pmatrix}
                      0 & (BC)^d \\
                      BC(BC)^d & -BC(BC)^dA(BC)^d
                    \end{pmatrix}^{n-j}\\
&\quad+\begin{pmatrix}
                 0 & (BC)^d \\
                 BC(BC)^d & -BC(BC)^dA(BC)^d
               \end{pmatrix}^n
\end{align*}
for $n\geq1$ and
$
\Delta=\begin{pmatrix}
                A & I \\
                BC &0
              \end{pmatrix}.
$
\end{theorem}
\begin{proof}
 Note that
 \begin{equation}\label{equ=pq}
   M=\begin{pmatrix}
         0 & A & I \\
         I & C & 0
       \end{pmatrix}
     \begin{pmatrix}
         0 & D \\
         I & 0 \\
         0 & B
       \end{pmatrix}.
 \end{equation}
Let us denote by $P$ and $Q$ the left matrix and the right matrix of the right hand side in \eqref{equ=pq}, respectively.
Then $QP=\begin{pmatrix}
   \alpha&\beta\\ \gamma&\Delta
\end{pmatrix}$, where
\begin{equation}\label{eqd1}
\alpha=D,~~~ \beta=\begin{pmatrix}
                   DC & 0
                 \end{pmatrix},~~~
\gamma=\begin{pmatrix}
         0 \\
         B
       \end{pmatrix}~~~~\text{and}~~~~
       \Delta=\begin{pmatrix}
                A & I \\
                BC &0
              \end{pmatrix}.
\end{equation}
Applying Corollary \ref{corNdn} to $\Delta$
we obtain the expression of $\Delta^{dn}$ for any $n\geq1$ as shown in Theorem \ref{BDd=0,FEFpi=0}.
Since $BD^d=0$ and $BD^iC=0$ for $i=1,2,\ldots$, we have
\begin{equation}\label{eqd2}
\alpha^n\beta=\begin{pmatrix}
                    D^{n+1}C & 0
                  \end{pmatrix},~~~~
\alpha^{dn}\beta=\begin{pmatrix}
                    D^{d(n+1)}C & 0
                  \end{pmatrix}
~~~~\text{and}~~~~
\gamma\alpha^n=\begin{pmatrix}
                 0 \\ BD^n
               \end{pmatrix}.
\end{equation}
Moreover,
$
\gamma\alpha^d=0 ~~~\text{and}~~~ \gamma\alpha^i\beta=0$ for $i=0,1,2,\ldots$.
Substitute \eqref{eqd1}, \eqref{eqd2} and $\Delta^{dn}$  into Lemma \ref{guoth3.2.2} to obtain $(QP)^d$.
By Lemma~\ref{cline} $M^{d}= P(QP)^{2d}Q $ and a routine computation we get the expression of $M^d$ as shown in Theorem \ref{BDd=0,FEFpi=0}.
%\begin{equation*}
%\begin{array}{ll}
%M^{d}=&\begin{pmatrix}
%
%
%
%0 &A & I \\
%U &C & 0
%\end{pmatrix}\begin{pmatrix}
%   V(D^{2d}+\Omega) U & V\Psi\\
%    \Phi U& \Delta^d
%    \end{pmatrix}^2
%\begin{pmatrix}
%0 & V \\
%I & 0 \\
% 0 & B
%\end{pmatrix}.\\
%\end{array}
%\end{equation*}
\end{proof}

We now analyse some special cases of the preceding theorem, some of which give results of
\cite{Castro2009,Deng2010,Djordjev2001,guophd,Mosic2014}.

\begin{corollary}
Let $M=\begin{pmatrix}
         A & B \\
         C & D
       \end{pmatrix}$ be an operator matrix
with $A$ and $D$ generalized Drazin invertible. If $A(BC)^\pi=0$, $(BC)^\pi BCA=0$, $BD=0$, then
  $M$ is generalized Drazin invertible, and
\begin{align*}
                          M^{d}=
    \begin{pmatrix}
       0 &0\\
       0 &D^d
     \end{pmatrix}
 +
    \begin{pmatrix}
       0 \\
       DD^d
     \end{pmatrix}\Psi\begin{pmatrix}
                          I & 0 \\
                          0 & B
                        \end{pmatrix}
 +
    \begin{pmatrix}
       0 \\
       D\Psi
     \end{pmatrix}
     \Delta^d\begin{pmatrix}
                          I & 0 \\
                          0 & B
                        \end{pmatrix}
     +\begin{pmatrix}
                          A & I \\
                          C & 0
                        \end{pmatrix}\Delta^{2d}
    \begin{pmatrix}
                          I & 0 \\
                          0 & B
                        \end{pmatrix},
    \end{align*}
where
\begin{align*}
&\Psi=\sum_{k=0}^\infty\begin{pmatrix}
D^{k}D^\pi C & 0
\end{pmatrix}\Delta^{d(k+2)}
+\sum_{k=0}^\infty\begin{pmatrix}
D^{d(k+2)}C & 0
\end{pmatrix}\Delta^k\Delta^\pi-\begin{pmatrix}
D^dC & 0
\end{pmatrix}\Delta^d,\\
&\Delta^{dn}=
\sum\limits_{i=0}^\infty\begin{pmatrix}
                                         0 & (BC)^\pi A(BC)^d  \\
                                           0& 0
                                       \end{pmatrix}\begin{pmatrix}
                                                           0 & (BC)^d \\
                                                          I & -A(BC)^d
                                                         \end{pmatrix}^{2i+n}\\
&\qquad\qquad+\begin{pmatrix}
                 0 & (BC)^d \\
                 BC(BC)^d & -BC(BC)^dA(BC)^d
               \end{pmatrix}^n
\end{align*}
for $n\geq1$ and
$
\Delta=\begin{pmatrix}
                A & I \\
                BC &0
              \end{pmatrix}.
$
\end{corollary}
\begin{proof}
 The result can be deduced by routine computations.
\end{proof}

The corollary above relaxes and removes Theorem 2.3 of \cite{Mosic2014}, in which Mosi\'{c} consider the conditions
$BD=0$, $A(BC)^\pi=0$, $C(BC)^\pi=0$, and $(BC)^\pi B=0$.

In \cite{Castro2009,Deng2010,Djordjev2001,guophd},
expressions of the GD-inverse of $M$ are given under the following conditions
\begin{enumerate}
\item $BC =0, BD = 0$ and $DC = 0$ (see \cite{Djordjev2001});
\item $BC =0, BD = 0$ (see \cite{Deng2010});
\item $BCA=0,BD=0,DC=0$ (see \cite{Castro2009});
\item $BCA=0,BD=0,$ and $BC$ is nilpotent (see \cite{Castro2009});
\item $BD^d=0$, $BD^iC=0$ for $i=0,1,...,n-1$ (see \cite{guophd}).
\end{enumerate}
Theorem \ref{BDd=0,FEFpi=0} relaxes some conditions in each item of (1)-(5) and gives a unified generalization of
\cite[Theorem 5.3]{Djordjev2001}, \cite[Theorem 2]{Deng2010}, \cite[Theorem 4.4]{Castro2009}, \cite[Theorem 4.2]{Castro2009},
and \cite[Theorem 3.2.1]{guophd}

We conclude this paper with some remarks. Using a way similar to Theorem \ref{BDd=0,FEFpi=0} we can give an expression of the generalized Drazin inverse $M^d$ under the following condition
$$(BC)^\pi A(BC)^d=0,~~ABC(BC)^\pi=0,~~BD^d=0, ~~~BD^iC=0 ~~\forall i\geq 1,$$
 which gives a unified generalization of
\cite[Theorem 5.3]{Djordjev2001}, \cite[Theorem 2]{Deng2010}, \cite[Theorem 1]{CvetkovAS2011}, \cite[Theorem 3.2.5]{guophd} and \cite[Theorem 2.5]{Mosic2014}.

Moreover, by using \cite[Theorem 3.2.4]{guophd} instead of Lemma \ref{guoth3.2.2} and by using the similar argument  we can give expressions of the generalized Drazin inverse $M^d$ under the following conditions, respectively,
\begin{enumerate}
\item $(BC)^dA(BC)^\pi=0$, $(BC)^\pi BCA=0$, $D^dC=0$, and $BD^iC=0 ~~\forall i\geq 1$;\\
\item $(BC)^\pi A(BC)^d=0$, $ABC(BC)^\pi=0$, $D^dC=0$, and $BD^iC=0~~\forall i\geq 1$.
\end{enumerate}
%
%
%\begin{lemma} \cite[Theorem 3.2.4]{guophd}\label{guoth3.2.4}
%Let
%$x=\begin{pmatrix}
%      a & b \\
%      c & d
%    \end{pmatrix}\in\mathcal {B}_{2\times2}$
%with $a$ and $d$ generalized Drazin invertible. If $a^db=0,ca^ib=0$ for any nonnegative integer $i$, then
%  $x$ is generalized Drazin invertible, and
% $$x^d=\begin{pmatrix}
%         a^d+\overline{\varphi} & \overline{\psi} \\
%         \overline{\phi} & d^d
%       \end{pmatrix},
%  $$
%where
%\begin{equation*}
%\begin{array}{ll}
% \overline{\psi}=&\Sigma_{i=0}^\infty a^ibd^{d(i+2)},\\
%    \overline{\phi}=&d^\pi\Sigma_{i=0}^\infty d^ica^{d(i+2)}+\Sigma_{i=0}^\infty d^{d(i+2)}ca^ia^\pi-d^dca^d,\\
%    \overline{\varphi}=&\Sigma_{i=0}^\infty\Sigma_{j=0}^\infty a^ibd^{d(i+j+3)}ca^ja^\pi
%       +\Sigma_{i=0}^\infty\Sigma_{j=0}^ia^{i-j}bd^jca^{d(i+3)}\\
%       &-\Sigma_{i=0}^\infty\Sigma_{j=0}^\infty a^ibd^{d(i+2)}d^jca^{d(j+1)}.
%\end{array}
%\end{equation*}
%\end{lemma}
%The following items, which are all a dual version of Theorem \ref{BDd=0,FEFpi=0}, can be proved similarly:
%\begin{enumerate}
%\item $(BC)^eA(BC)^\pi=0,(BC)^\pi BCA=0,D^dC=0$, $BD^iC=0,i=1,2,\ldots$,\\
%\item $(BC)^\pi A(BC)^e=0,ABC(BC)^\pi=0,BD^d=0$, $BD^iC=0,i=1,2,\ldots$,\\
%\item $(BC)^\pi A(BC)^e=0,ABC(BC)^\pi=0,D^dC=0$, $BD^iC=0,i=1,2,\ldots$.
%\end{enumerate}
which give a unified generalization of
\cite[Theorem 5.3]{Djordjev2001}, \cite[Theorem 3]{Deng2010}, \cite[Theorem 4.4, Theorem 4.5]{Castro2009},
\cite[Theorem 3.2.3, Theorem 3.2.6]{guophd}, \cite[Theorem 2.4, Theorem 2.6]{Mosic2014}, and
\cite[Theorem 1--3]{CvetkovAS2011}.
%\cite[Theorem 5.3]{Djordjev2001}, \cite[Theorem 3]{Deng2010}, \cite[Theorem 4.4]{Castro2009}, \cite[Theorem 4.5]{Castro2009},
%\cite[Theorem 3.2.3]{guophd}, \cite[Theorem 2.4]{Mosic2014}, and
%\cite[Theorem 1--3]{CvetkovAS2011}, \cite[Theorem 3.2.6]{guophd}, \cite[Theorem 2.6]{Mosic2014}.

\bibliographystyle{amsplain}

\begin{thebibliography}{}
\bibitem{Ben2006}J. Ben\'{i}tez, N. Thome, The generalized Schur complement in group inverses and ($k$+1)-potent
matrices, Linear Multilinear Alg. 54 (2006) 405--413.
\bibitem{Campbell1980I}S.L. Campbell, Singular Systems of Differential Equations I, Pitman, San Francisco, CA, 1980.
\bibitem{Campbell1980II}S.L. Campbell, Singular Systems of Differential Equations II, Pitman, San Francisco, CA, 1980.
\bibitem{Campbell1983}S.L. Campbell, The Drazin inverse and systems of second order linear differential equations, Linear Multilinear Algebra. 14 (1983) 195--198.
\bibitem{Campbell1991}S.L. Campbell, C.D. Meyer Jr., Generalized Inverses of Linear Transformations, Pitman, London, 1979. Dover, New York, 1991.
\bibitem{Campbell1976}S.L. Campbell, C.D. Meyer Jr., N.J. Rose, Applications of the Drazin inverse to linear systems of differential equations, SIAM J. Appl. Math. 31 (1976) 411--425.
\bibitem{Castro2009}N. Castro-Gonz\'{a}lez, E. Dopazo, M.F. Matr\'{i}nez-Serrano, On the Drazin inverse of the sum of two operators and its application to operator matrices, J. Math. Anal. Appl. 350 (2009) 207--215.
\bibitem{Castro2006}N. Castro-Gonz\'{a}lez, E. Dopazo, J. Robles, Formulas for the Drazin inverse of special block
matrices, Appl. Math. Comput. 174 (2006) 252--270.
\bibitem{Castro2004}N. Castro-Gonz\'{a}lez, J.J. Koliha,  New additive results for the g-Drazin inverse,
 Proc. Roy. Soc. Edinburgh Sect. A  134  (2004) 1085--1097.
\bibitem{Castro2002}N. Castro-Gonz\'{a}lez, J.J. Koliha, V. Rako\v{c}evi\'{c}, Continuity and general perturbation of the Drazin inverse for closed linear operators, Abstr. Appl. Anal. 7 (2002) 335--347.
\bibitem{Cvetkov2008}D.S. Cvetkovi\'{c}-Ili\'{c}, J. Chen, Z. Xu, Explicit representations of the Drazin inverse of block matrix and modified matrix, Linear Multilinear Algebra 14(2008) 1--10.
\bibitem{CvetkovAS2011}A.S. Cvetkovi\'{c}, G. V. Milovanovi\'{c},  On Drazin inverse of operator matrices,
 J. Math. Anal. Appl.  375  (2011) 331--335.
\bibitem{Deng2010}C. Deng, D.S. Cvetkovi\'{c}-Ili\'{c}, Y. Wei, Some results on the generalized Drazin inverse of operator matrices, Linear Multilinear Algebra. 58 (2010) 503--521.
\bibitem{Djordjev2001}D.S. Djordjevi\'{c}, P.S. Stanimirovi\'{c}, On the generalized Drazin inverse and generalized
resolvent, Czechoslovak Math. J. 51(126) (2001) 617--634.
\bibitem{Dopazo2010}E. Dopazo, M.F. Matr¨ªnez-Serrano, Further results on the representation of the Drazin inverse of a $2\times2 $ block matrix, Linear Algebra Appl. 432 (2010) 1896--1904.
\bibitem{Du2005}H. Du, C. Deng, The representation and characterization of Drazin inverses of operators on
a Hilbert space, Linear Algebra Appl. 407 (2005) 117--124.
\bibitem{guophd}L. Guo, Representations of Generalized Drazin inverses for Operator Matrices, PHD Dissertation,
 Univ. of Jilin, Changchun 2010.
\bibitem{Guo2010}L. Guo, X. Du, Representations for the Drazin inverses of $2\times2$ block matrices, Appl. Math. Comput. 217 (2010) 2833--2842.
\bibitem{Hartwig1977}R.E. Hartwig, J.M. Shoaf, Group inverses and Drazin inverses of bidiagonal and triangular Toeplitz matrices, J. Austral. Math. Soc. 24 (1977) 10--34.
\bibitem{Hartwig2006}R.E. Hartwig, X. Li, Y. Wei, Representations for the Drazin inverse of a $2\times2$ block matrix, SIAM J. Matrix Anal. Appl. 27 (2006) 757--771.
\bibitem{Koliha1996}J.J. Koliha, A generalized Drazin inverse, Glasg. Math. J. 38 (1996) 367--381.
\bibitem{cline2014} Y. Liao, J. Chen, J. Cui, Cline's formula for the generalized Drazin inverse,
 Bull. Malays. Math. Sci. Soc. (2)  37  (2014)  37--42.
\bibitem{Liu2012}X. Liu, H Yang, Further results on the group inverses and Drazin inverses of
 anti-triangular block matrices, Appl. Math. Comput, 218 (2012) 8978--8986.
\bibitem{Meyer1997}C.D. Meyer Jr., N.J. Rose, The index and the Drazin inverse of block triangular matrices, SIAM J. Appl. Math. 33 (1977) 1--7.
\bibitem{Mosic2014}D. Mosi\'{c}, Group inverse and generalized Drazin inverse of block matrices in Banach algebra,
Bull. Korean Math. Soc. 51 (2014) 765--771.
\bibitem{Wei2005}Y. Wei, H. Diao, On group inverse of singular Toeplitz matrices, Linear Algebra Appl. 399
(2005) 109--123.
\bibitem{Wei2006(1)}Y. Wei, H. Diao, M.K. Ng, On Drazin inverse of singular Toeplitz matrix, Appl. Math.
Comput. 172 (2006) 809--817.
\bibitem{Wei2006(2)}Y. Wei, X. Li, F. Bu, F. Zhang, Relative perturbation bounds for the eigenvalues of
diagonalizable and singular matrices -- application of perturbation theory for simple invariant
subspaces, Linear Algebra Appl. 419 (2006) 765--771.
\end{thebibliography}

\end{document}